\documentclass[12pt]{article}

\usepackage{amsmath,amssymb,amsthm,amsfonts,diagrams,graphicx,setspace,float}

\def\spn{\mathop{\rm span}} \def\Hom{{\rm Hom}} \def\ind{\mathop{\rm
    ind}} \def\ad{\mathop{\rm ad}} \def\im{\mathop{\rm im}}
\def\ind{{\rm ind}}  \def\phi{\varphi}
\def\g{\mathfrak g} \def\m{\mathfrak m} \def\V{\mathfrak V}
\def\F{\mathbb F} \def\Z{\mathbb Z} \def\N{\mathbb N}

\makeatletter \let\@@pmod\pmod
\DeclareRobustCommand{\pmod}{\@ifstar\@pmods\@@pmod}
\def\@pmods#1{\mkern4mu({\operator@font mod}\mkern 6mu#1)}
\makeatother

\newtheorem{theorem}{Theorem}[section]
\newtheorem{lemma}[theorem]{Lemma}
\newtheorem{proposition}[theorem]{Proposition}
\newtheorem*{remark}{{\bf Remark}}
\providecommand{\keywords}[1]{\noindent{Keywords:} #1}
\providecommand{\classify}[1]{\noindent{Mathematics Subject
    Classification:} #1}
\renewcommand{\footnotesize}{\scriptsize}

\title{Restricted One-dimensional Central Extensions of the Restricted
  Filiform Lie Algebras $\m_0^\lambda(p)$}

\author{Tyler J. Evans \\Department of Mathematics \\Humboldt State
  University \\Arcata, CA 95521 USA \\evans@humboldt.edu \and Alice
  Fialowski\footnote{The final version of the paper was written during
    the stay of the second author at the Max-Planck Institute for
    Mathematics Bonn, Germany.} \\Institute of Mathematics
  \\University of P\'ecs and E\" otv\" os Lor\' and University \\
  Hungary\\ fialowsk@ttk.pte.hu, fialowsk@cs.elte.hu}

\date{}

\begin{document}

\maketitle

\keywords{restricted Lie algebra; central extension; cohomology;
  filiform Lie algebra}

\classify{17B50, 17B56}

\begin{abstract}
  Consider the filiform Lie algebra $\m_0$ with nonzero Lie brackets
  $[e_1,e_i]=e_{i+1}$ for $1<i<p$, where the characteristic of the
  field $\F$ is $p > 0$. We show that there is a family
  $\m_0^{\lambda}(p)$ of restricted Lie algebra structures
  parameterized by elements $\lambda \in \F^p$. We explicitly describe
  both the ordinary and restricted $1-$cohomology spaces and show that
  for $p \ge3$ these spaces are equal. We also describe the ordinary
  and restricted $2-$cohomology spaces and interpret our results in the
  context of one-dimensional central extensions.
\end{abstract}

\section{Introduction} In the past years a lot of attention has been
paid to nilpotent $\N$-graded Lie algebras of maximal class. By
$\N$-graded we mean that the Lie algebra is the direct sum of
subspaces $\g_i$, $i \in \N$ such that $[\g_i,\g_j] \subset
\g_{i+j}$. A Lie algebra of \emph{maximal class} is a graded Lie
algebra
$$
\g=\oplus_{i=1}^{\infty}\g_i
$$
over a field $\F$, where $\dim(\g_1)=\dim(\g_2)=1$, dim$(\g_i) \leq 1$
for $i \geq 3$ and $[\g_i,\g_1]= \g_{i+1}$ for $i \geq 1$.  Algebras
of maximal class are either finite dimensional, or all their proper
factors are finite dimensional. These infinite dimensional algebras
can be viewed as (projective) limits of nilpotent Lie algebras of
maximal class. They are also called \emph{narrow, thin}, or
\emph{coclass 1} Lie algebras.

If the number of generators of such Lie algebra is the minimal $2$,
then all such Lie algebras are classified in characteristic $0$.  In
fact, all $\N$-graded infinite dimensional Lie algebras with two
generators $e_1$ and $e_2$ have been classified in \cite{Fi}, and
exactly 3 of those are of \emph{filiform} type. An $n$-dimensional
nilpotent Lie algebra is filiform, if dim$\g^2=n-2$, ...,
dim$\g^k=n-k$, ..., dim$\g^n=0$ where $\g^k=[\g,\g^{k-1}]$,
$2 \leq k \leq n$.  We also call their projective limit filiform
type. We list them with the nontrivial bracket structures:

\begin{align*}
  \m_0:& \quad [e_1 ,e_i]= e_{i+1},\qquad  i \ge 2, \quad i \in \N\\
  \m_2:& \quad [e_1, e_i]= e_{i+1}, \qquad i \ge 2,\quad  [e_2, e_j] = e_{j+2},
         \quad  j \ge 3, \quad i,j \in \N\\
  \V:& \quad    [e_i, e_j] = (j-i)e_{i+j}, \qquad i, j \geq 1.
\end{align*} 

In the finite-dimensional case in characteristic zero, the
classification of finite-dimensional $\N$-graded filiform Lie algebras
is also done in \cite{Mill}. They include the natural ``truncations''
of the above three algebras $\m_0(n), \m_2(n)$ and $\V(n)$, which are
obtained by taking the quotient by the ideal generated by $e_{n+1}$,
but there are other families as well.

The picture is more complicated in positive characteristic, see
\cite{CMN, CN, J}, but $\m_0$, $\m_2$ and their truncations always
show up.

The cohomology of $\N$-graded Lie algebras of maximal class has been
studied extensively over a field of characteristic zero (\cite{Fi,
  FiMi, Ve}), but for positive characteristic, much less is
known. Recently it was shown that over a field of characteristic two,
the algebras $\m_0(n)$ and $\m_2(n)$ have the same Betti numbers
\cite{Ts}, and the cohomology spaces with trivial coefficients are
obtained in this case. In fact, these cohomology spaces are isomorphic
\cite{NiTs}. For the truncated finite dimensional analogues, the first
3 cohomology spaces are known over $\Z_2$.

In this paper we show if the field $\F$ has characteristic $p>0$, the
Lie algebra $\m_0(p)$ admits the structure of a restricted Lie
algebra. In fact, we show that there is a family $\m_0^\lambda(p)$ of
such restricted Lie algebra structures parameterized by elements
$\lambda\in\F^p$. Using the ordinary Chevalley-Eilenberg complex and
the partial restricted complex in \cite{EvansFuchs2}, we calculate the
dimensions of both the ordinary cohomology $H^q(\m_0^\lambda(p))$ and
the restricted cohomology $H^q_*(\m_0^\lambda(p))$ for $q=1,2$, and we
explicitly describe bases for these spaces.

The organization of the paper is as follows. Section 2 contains the
constructions of the Lie algebras and restricted Lie algebras under
consideration including explicit descriptions of the Lie brackets,
$[p]$-operators and both ordinary and restricted cochain spaces and
differentials. Section 3 contains the computation of both the ordinary
cohomology $H^1(\m_0^\lambda(p))$ and restricted cohomology
$H^1_*(\m_0^\lambda(p))$, and in particular it is shown that these
spaces are equal for primes $p\ge 3$. Section 4 contains the
computations of $H^2(\m_0^\lambda(p))$ and
$H^2_*(\m_0^\lambda(p))$. In Section 5, we interpret our results in
the context of one-dimensional (both ordinary and restricted) central
extensions.

\paragraph {\bf Acknowledgements.}
The authors are grateful for helpful conversations with Dmitry
B. Fuchs, and for useful suggestions from the referee.

\section{Preliminaries}

\subsection{The Lie Algebra $\m_0(p)$ over $\F$}

Let $p > 0$ be a prime, and let $\F$ be a field of characteristic
$p$. Define the $\F$-vector space
\[\m_0(p)=\spn\nolimits_\F(\{e_1,\dots , e_p\}), \]
and define a bracket on $\m_0(p)$ by
\[[e_1,e_i]=e_{i+1},\quad 1<i < p,\] with all other brackets
$[e_i,e_j]=0$.  Note that $\m_0(p)$ is a graded Lie algebra with $k$th
graded component $(\m_0(p))_k=\F e_k$ for $1\le k\le p$.  If
$\alpha_i, \beta_i \in \F$ and $g=\sum_{i=1}^p \alpha_ie_i$,
$h=\sum_{i=1}^p \beta_ie_i$, then
\begin{eqnarray}\label{bracket} [g,h]=\sum_{j=3}^p
  (\alpha_1\beta_{j-1}-\alpha_{j-1}\beta_1) e_j.
\end{eqnarray}

\subsection{The Restricted Lie Algebras $\m_0^\lambda(p)$}

For any $j\geq 2$ and $g_1,\dots, g_j\in\m_0(p)$, we write the
$j$-fold bracket
\[[g_1,g_2,g_3,\dots, g_j]=[[\dots[[g_1,g_2],g_3],\dots,]g_j].\]
Equation (1) implies that the center of the algebra is
$Z(\m_0(p))= \F e_p,$ and that any $p$-fold bracket
$[g_1,g_2,g_3,\dots, g_p]=0.$ In particular, equation (1) implies
$(\ad g)^p=0$ for all $g\in\m_0(p)$. Therefore $(\ad e_k)^p =0$ is an
inner derivation for all $k$ so that $\m_0(p)$ admits the structure of
a restricted Lie algebra (see \cite{Jac}). To define a $[p]$-operator
on $\m_0(p)$, we choose for each $1\le k\le p$, an element $e_k^{[p]}$
such that
\[\ad e_k^{[p]} = (\ad e_k)^p=0.\]
That is, we must choose $e_k^{[p]}$ in the center $\F e_p$ of the
algebra $\m_0(p)$. If we let
$\lambda=(\lambda_1, ..., \lambda_p)\in\F^p$, then setting
$e_k^{[p]}=\lambda_k e_p$ for each $k$ defines a restricted Lie
algebra that we denote by $\m_0^\lambda(p)$.  Because $p$-fold
brackets are zero, if $[p]$ is any restricted Lie algebra operator on
$\m_0(p)$, then for all $g,h\in\m_0(p)$, $\alpha\in\F$,
\[(g+h)^{[p]}=g^{[p]}+h^{[p]}\ \ \mbox{\rm and}\ \ (\alpha
  g)^{[p]}=\alpha^p g^{[p]}.\] It follows that if
${\bf \lambda}\in\F^p$, then for all
$g=\sum \alpha_k e_k\in\m_0^\lambda(p)$,
\begin{eqnarray}\label{p-op} g^{[p]}=\left (\sum_{k=1}^p
  \alpha_k^p\lambda_k \right ) e_p.
\end{eqnarray}
Everywhere below, we write $\m_0^\lambda(p)$ to denote both the Lie
algebra $\m_0(p)$ and the restricted Lie algebra $\m_0^\lambda(p)$ for
a given $\lambda\in\F^p$.

A natural question arises: For which $\lambda,\lambda'\in\F^p$ are the
graded restricted Lie algebras $\m_0^\lambda(p)$ and
$\m_0^{\lambda'}(p)$ isomorphic?

\begin{proposition}
  If $\lambda,\lambda'\in\F^p$ , the graded restricted Lie algebras
  $\m_0^{\lambda}(p)$ and $\m_0^{\lambda'}(p)$ are isomorphic if and
  only if $\lambda_1=\mu_1\lambda'_1$ and $\lambda_2=\mu_2\lambda_2'$
  where $\mu_1, \mu_2 \in \F$ are independent parameters, and
  \[
    \lambda_k=\mu_2^{p-1}\mu_1^{p(k-3)+2}\lambda'_k,
  \]
  for $k \in 3,...,p$.

\end{proposition}

\begin{proof}
  Assume that there exists a graded restricted Lie algebra isomorphism
  $\phi: \m_0^{\lambda}(p) \to \m_0^{\lambda'}(p)$, and let
  $\phi(e_1)=\mu_1 e_1, \phi(e_2)=\mu_2 e_2$ for some
  $\mu_1, \mu_2 \in \F$. Since $\phi$ preserves the Lie bracket, we
  must have $\phi(e_k)=\mu_2\mu_1^{k-2} e_k$, $k=3,\dots, p$. Let
  $\mu_k=\mu_2\mu_1^{k-2}$ for $3\le k\le p$ so that
  $\phi(e_k)=\mu_ke_k$ for all $k$.  Moreover, $\phi$ preserves the
  restricted $[p]$-structure so that
  \[
    \phi(e_k^{[p]})=\phi(e_k)^{[p]'}
  \]
  for all $k$ (here $[p]'$ denotes the restricted $[p]$-structure on
  $\m_0^{\lambda'}(p)$).  Now,
  \[
    \phi(e_k^{[p]})=\phi(\lambda_k e_p)=\lambda_k\mu_p e_p\ \mbox{\rm
      and}\ \phi(e_k)^{[p]'}=(\mu_k e_k)^{[p]'}=\mu_k^p\lambda_k'e_p
  \]
  so that $\lambda_k\mu_p=\mu_k^p\lambda'_k,$ and hence
  $\lambda_k=\mu_k^p\mu_p^{-1}\lambda'_k$.  But
  $\mu_k=\mu_2 \mu_1^{k-2}$, so that
  \[
    \lambda_k=\mu_2^p\mu_1^{p(k-2)}\mu_2^{-1}\mu_1^{2-p}\lambda'_k =
    \mu_2^{p-1}\mu_1^{p(k-3)+2}\lambda'_k.
  \]
  It remains to show that the above condition on the $\lambda_k$ gives
  rise to a graded restricted Lie algebra isomorphism between
  $\m_0^{\lambda}(p)$ and $\m_0^{\lambda'}(p)$. If, for
  $\mu_1,\mu_2\in\F$, we define
  $\phi(e_1)=\mu_1e_1, \phi(e_2)=\mu_2e_2$ and
  $\phi(e_k)=\mu_2\mu_1^{k-2}e_k\ (3\le k\le p)$, it is easy to check
  the argument above is reversible, and we obtain a graded isomorphism
  between the restricted Lie algebras.

\end{proof}

\subsection{Cochain Complexes with Trivial Coefficients}

\subsubsection{Ordinary Cochain Complex}

For ordinary Lie algebra cohomology with trivial coefficients, the
relevant cochain spaces (with bases) are:
\begin{align*}
  C^0(\m_0^\lambda(p))&= \F, & \{1\};\\
  C^1(\m_0^\lambda(p))&= \m_0^\lambda(p)^*, & \{e^k\ |\ 1\le k\le p\};\\
  C^2(\m_0^\lambda(p))&= (\wedge^2\m_0^\lambda(p))^*,& \{e^{i,j}\ |\ 1\le i<j\le p\};\\
  C^3(\m_0^\lambda(p))&= (\wedge^3\m_0^\lambda(p))^*, & \{e^{s,t,u}\ |\ 1\le s<t<u\le p\},\\
\end{align*}
and the differentials are defined by:
\begin{align*}
  d^0: C^0(\m_0^\lambda(p))\to C^1(\m_0^\lambda(p))  &\ \ \ \  d^0=0;\\
  d^1:C^1(\m_0^\lambda(p))\to C^2(\m_0^\lambda(p)) &\ \ \ \  d^1(\psi)(g,h)=\psi([g,h]);\\
  d^2:C^2(\m_0^\lambda(p))\to C^3(\m_0^\lambda(p)) &\ \ \ \  d^2(\phi)(g,h,f)=\phi([g,h]\wedge f)-\phi([g,f]\wedge h)+\phi([h,f]\wedge g).\\
\end{align*}
The cochain spaces $C^n(\m_0^\lambda(p))$ are graded:
\begin{align*}
  C^1_k(\m_0^\lambda(p))&=\spn(\{e^k\}),& 1\le k\le p;\\
  C^2_k(\m_0^\lambda(p))&=\spn(\{e^{i,j}\}),&  1\le i<j \le p, i+j=k, 3\le k\le 2p-1;\\
  C^3_k(\m_0^\lambda(p))&=\spn(\{e^{s,t,u}\}),&  1\le s<t<u \le p, s+t+u=k, 6\le k\le 3p-3,\\
\end{align*}
and the differentials are graded maps. For $1\le k\le p$, if we write
\[d^1(e^k)=\sum_{1\le i<j\le p} A^k_{ij}e^{i,j},\] then for
$1\le q <r\le p$, we have
\[A^k_{qr}=d^1(e^k)(e_{qr})=e^k[e_q,e_r]=\left\{
    \begin{array}{ll}
      0& k=1,2;\\
      1 & q=1, r=k-1, k\ge 3\\
      0 & \mbox{\rm otherwise.}
    \end{array}\right.\]
Therefore $d^1(e^1)=d^1(e^2)=0$ and for $3\le k\le p$,
\begin{eqnarray}\label{1-coboundary}
  d^1(e^k)=e^{1,k-1}.
\end{eqnarray}
For $1\le i<j\le p$, if we write
\[d^2(e^{i,j})=\sum_{1\le s<t<u\le p} A^{ij}_{stu}e^{s,t,u}.\] Then
for $1\le l<m<n\le p$, we have
\begin{align}
  \begin{split}
    \label{eq:3}
    A^{ij}_{lmn} &= d^2(e^{i,j})(e_{lmn})\\
    &=e^{i,j}([e_l,e_m]\wedge e_n)-e^{i,j}([e_l,e_n]\wedge
    e_m)+e^{i,j}([e_m,e_n]\wedge e_l).
  \end{split}
\end{align}
Therefore (\ref{eq:3}) shows $d^2(e^{1,j})=0$ for $2\le j\le p$, and
for $i\ge 2$,
\begin{eqnarray}
  \label{2-coboundary}
  d^2(e^{i,j})=e^{1,i-1,j}+e^{1,i,j-i}.
\end{eqnarray}

\subsubsection{Restricted Cochain Complex}

For convenience, we include a brief description of the (partial)
restricted cochain complex employed below to compute the restricted
Lie algebra cohomology with trivial coefficients.  We refer the reader
to \cite{EvansFuchs2} or \cite{EvFiPe} for a detailed description of
this (partial) complex. The first two restricted cochain spaces
coincide with the ordinary cochain spaces:
\begin{align*}
  C^0_*(\m_0^\lambda(p))&= C^0(\m_0^\lambda(p))= \F\\
  C^1_*(\m_0^\lambda(p))&=   C^1(\m_0^\lambda(p))=\m_0^\lambda(p)^*.
\end{align*}
Using the same notation as in \cite{Viv}, we let $\Hom_{\rm Fr}(V,W)$
denote the set of {\it Frobenius homomorphisms} from the $\F$-vector
space $V$ to the $\F$-vector space $W$. That is
\[ \Hom_{\rm Fr}(V,W) = \{ f:V\to W\ |\ f(\alpha x+ \beta y) =
  \alpha^p f(x) + \beta^p f(y)\}\] for all $\alpha,\beta\in\F$ and
$x,y\in V$.

If $\phi\in C^2(\m_0^\lambda(p))$ and
$\omega\in\Hom_{\rm Fr}(\m_0^\lambda(p),\F)$, then we say $\omega$ has
the {\bf $*$-property} with respect to $\phi$ if for all
$g,h\in \m_0^\lambda(p)$ we have
\begin{equation}
  \label{starprop}
  \omega (g+h)=\omega (g)+ \omega (h) + \sum_{\substack{g_i=g\ {\rm
        or}\ h\\ g_1=g,g_2=h}} \frac{1}{\#(g)}
  \phi([g_1,g_2,\dots, g_{p-1}]\wedge g_p).
\end{equation} Here $\#(g)$ is the number
of factors $g_i$ equal to $g$. We remark that $\omega$ has the
$*$-property with respect to $\phi=0$ precisely when
$\omega\in\Hom_{\rm Fr}(\m_0^\lambda(p),\F)$. Moreover, given $\phi$, we can assign the
values of $\omega$ arbitrarily on a basis for $\m_0^\lambda(p)$ and use
(\ref{starprop}) to define $\omega\in\Hom_{\rm
  Fr}(\m_0^\lambda(p),\F)$ that has the $*$-property with respect to
$\phi$. We then define the space of restricted $2$-cochains as
\begin{align*}
  C^2_*(\m_0^\lambda(p))=\{(\phi,\omega)\ |\ &\phi\in C^2(\m_0^\lambda(p)), \omega\in\Hom_{\rm Fr}(\m_0^\lambda(p),F)\\
                                             & \mbox{\rm has the $*$-property with respect to $\phi$}\}.
\end{align*}
If $\alpha\in C^3(\m_0^\lambda(p))$ and
$\beta\in\m_0^\lambda(p)^*\otimes\Hom_{\rm Fr}(\m_0^\lambda(p),\F)$,
we say that $\beta$ has the {\bf $**$-property with respect to
  $\alpha$,} if for all $g,h_1,h_2\in\m_0^\lambda(p)$
\begin{align}\label{starstarprop}
  \beta(g,h_1+h_2) &=
                     \beta(g,h_1)+\beta(g,h_2)-\nonumber \\
                   & \sum_{\substack{l_1,\dots,l_p=1 {\rm or} 2\\ l_1=1,
  l_2=2}}\frac{1}{\#\{i_i=1\}}\alpha (g\wedge
  [h_{l_1},\cdots,h_{l_{p-1}}]\wedge h_{l_{p}}).
\end{align}
Again we remark that $\beta$ has the $**$-property with respect to
$\alpha=0$ precisely when
$\beta\in\m_0^\lambda(p)^*\otimes\Hom_{\rm
  Fr}(\m_0^\lambda(p),\F)$. Given $\alpha$, we can define the values
of $\beta$ arbitrarily on a basis and use (\ref{starstarprop}) to
define
$\beta\in\m_0^\lambda(p)^*\otimes\Hom_{\rm Fr}(\m_0^\lambda(p),\F)$
that has the $**$-property with respect to $\alpha$. We then define
the space of restricted $3$-cochains by
\begin{align*}
  C^3_*(\m_0^\lambda(p))=\{(\alpha,\beta)\ |\ &\alpha\in C^3(\m_0^\lambda(p)), \beta\in\m_0^\lambda(p)^*\otimes\Hom_{\rm Fr}(\m_0^\lambda(p),\F)\\ & \mbox{\rm has the $**$-property w.r.t. $\alpha$}\}.
\end{align*}
We will use the following bases for the restricted cochains:
\begin{align*}
  C^0_*(\m_0^\lambda(p))& & \{1\};\\
  C^1_*(\m_0^\lambda(p))& & \{e^k\ |\ 1\le k\le p\};\\
  C^2_*(\m_0^\lambda(p))& & \{(e^{i,j},\tilde e^{i,j})\ |\ 1\le i<j\le p\}\cup
                            \{(0,\overline e^k)\ |\ 1\le k\le p\} ;
\end{align*}
where $\overline e^k:\m_0^\lambda(p)\to\F$ is defined by
\[\overline e^k\left (\sum_{i=1}^p \alpha_ie_i\right ) = \alpha_k^p,\]
and $\tilde e^{i,j}$ is the map $\tilde e^{i,j}:\m_0^\lambda(p)\to \F$
that vanishes on the basis and has the $*$-property with respect to
$e^{i,j}$. More generally, given $\phi\in C^2(\m_0^\lambda (p))$, we
let $\tilde\phi: \m_0^\lambda (p) \to\F$ be the map that vanishes on
the basis for $\m_0^\lambda (p) $ and has the $*$-property with
respect to $\phi$. The restricted differentials are defined by
\begin{align*}
  d^0_*: C^0_*(\m_0^\lambda(p))\to C^1_*(\m_0^\lambda(p))  &\ \ \ d^0_*=0\\
  d^1_*:C^1_*(\m_0^\lambda(p))\to C^2_*(\m_0^\lambda(p)) &\ \ \
                                                           d^1_*(\psi)=(d^1(\psi),\ind^1(\psi))\\
  d^2_*:C^2_*(\m_0^\lambda(p))\to C^3_*(\m_0^\lambda(p)) &\ \ \ d^2_*(\phi,\omega)=(d^2(\phi),\ind^2(\phi,\omega))
\end{align*}
where $\ind^1(\psi)(g) := \psi(g^{[p]})$ and
$\ind^2(\phi,\omega)(g,h) := \phi(g\wedge h^{[p]})$.

If $\psi\in C^1_*(\m_0^\lambda(p))$ and
$(\phi,\omega)\in C^2_*(\m_0^\lambda(p))$, then $\ind^1(\psi)$ has the
$*$-property with respect to $d^1(\psi)$ and $\ind^2(\phi,\omega)$ has
the $**$-property with respect to $d^2(\phi)$ \cite{EvansFuchs2}.  If
$g=\sum\alpha_i e_i$, $h=\sum\beta_i e_i$, $\psi=\sum \mu_ie^i$ and
$\phi=\sum \sigma_{ij}e^{i,j}$, then

\begin{eqnarray}\label{induced1}
  \ind^1(\psi)(g) = \mu_p \left (\sum_{j=1}^p
  \alpha_j^p\lambda_j \right )
\end{eqnarray}
and
\begin{eqnarray}\label{induced2}
  \ind^2(\phi,\omega)(g,h) =  \left (\sum_{i=1}^p
  \beta_i^p\lambda_i \right )\left (\sum_{j=1}^{p-1} \alpha_j\sigma_{jp}\right).
\end{eqnarray}

\begin{remark}
  For a given $\phi\in C^2(\m_0^\lambda(p))$, if
  $(\phi,\omega),(\phi,\omega')\in C^2_*(\m_0^\lambda(p))$, then
  $d^2_*(\phi,\omega)=d^2_*(\phi,\omega')$. In particular, with
  trivial coefficients, $\ind^2(\phi,\omega)$ depends only on $\phi$.
\end{remark}

\section{The Cohomology $H^1(\m_0^\lambda(p))$ and
  $H^1_*(\m_0^\lambda(p))$}

\begin{theorem}\label{1-coho}
  If $p\ge 3$ and ${\bf \lambda} \in \F^p$,
  then \[H^1(\m_0^\lambda(p))=H^1_*(\m_0^\lambda(p))\] and the classes
  of $\{e^1,e^2\}$ form a basis.
\end{theorem}

\begin{proof}
  It follows easily from (\ref{1-coboundary}) that $\dim(\ker d^1)=2$
  and $\{e^1,e^2\} $ is a basis for this kernel. Moreover $d^0=0$, so
  that
  \[H^1(\m_0^\lambda (p))\cong\ker d^1=\F e^1\oplus \F e^2.\]

  Now, $H^1_*(\m_0 ^\lambda (p))$ consists of those ordinary
  cohomology classes $[\psi]\in H^1 (\m_0 ^\lambda (p))$ for which
  $\ind^1(\psi)=0$ \cite{EvansFuchs2}. If
  $\psi =\sum_{k=1}^p \mu_ke^k$ is any ordinary cocycle, then
  $\mu_p=0$ ($p\ge 3$) so that for any $g\in \m_0^\lambda (p)$, we
  have
  \[\ind^1(\psi)(g)=\psi(g^{[p]})=\mu_p\left (\sum_{k=1}^p
      \alpha_k^p\lambda_k\right )=0\] and hence
  $H^1_*(\m_0^\lambda(p))=H^1(\m_0(p))$.
\end{proof}

\begin{remark}
  For $p\ge 3$, formula (\ref{1-coboundary}) shows that for
  $3\le k\le p$, $d^1(e_k)=e^{1,k-1}$, so the set
  $\{e^{1,2},\dots, e^{1,p-1}\}$ is a basis for the image
  $d^1(C^1(\m_0^\lambda(p))$.
\end{remark}

The above calculation of the ordinary cohomology
$H^1(\m_0^\lambda(p))$ is valid also in the case $p=2$, but the
restricted cohomology depends on $\lambda$ in this case. In
particular, if $p=2$, and $\lambda=(0,0)$, then $\ind^1(\psi)(g)=0$
for all $g\in\m_0$ and hence
$H^1_*(\m_0^\lambda(2))=H^1(\m_0^\lambda(2))$. If $\lambda\ne(0,0)$,
then $\ind^1(\psi)(g)=0$ for all $g\in\m_0$ if and only if $\mu_2=0$
so that $\ker d^1_*=\{e^1\}$ and $H^1_*(\m_0^\lambda(2))$ is one
dimensional.

\section{The Cohomology $H^2(\m_0^\lambda(p))$ and
  $H^2_*(\m_0^\lambda(p))$}

\subsection{Ordinary Cohomology}
\begin{theorem}\label{ordinary2}
  If $p=2$,
  $H^2(\m_0^\lambda(2))\cong C^2(\m_0^\lambda(2))=\spn(\{e^{1,2}\})$
  is 1-dimensional.
  
  If $p\ge 3$, then
  \[\dim(H^2(\m_0^\lambda(p)))=\frac{p+1}{2}\] and the
  cohomology classes of the cocycles
  $\{e^{1,p},\phi_5,\phi_7,\phi_9,\dots, \phi_{p+2}\}$ form a basis,
  where
  \[\phi_k=e^{2, k-2}-e^{3, k-3}+\cdots +(-1)^{\lfloor
      \frac{k}{2}\rfloor} e^{\lfloor \frac{k}{2}\rfloor, k-\lfloor
      \frac{k}{2}\rfloor}.\]
\end{theorem}

\begin{proof} If $p=2$, the algebra $\m_0^\lambda(2)$ is abelian so
  that $d^1=d^2=0$.

  If $p\ge 3$, the proof of Theorem~\ref{1-coho} and the remark
  following the proof show that we have $\dim(\im d^1)=p-2$ and
  $\{e^{12},\dots, e^{1p-1}\}$ is a basis for this image. If $p=3$,
  $d^2=0$ so that the ordinary cohomology $H^2(\m_0^\lambda(3))$ is
  2-dimensional and has a basis consisting of the classes of the
  cocyles $e^{13}$ and $\phi_5=e^{23}$.

  If $p > 3$, a basis for $\ker d^2$ is
  \begin{align*}
    B(Z^2) = \{e^{1,2}, e^{1,3}, ..., e^{1,p}, \phi_5, \phi_7, \dots , \phi_{p+2}\}
  \end{align*}
  where
  $\phi_k=e^{2, k-2}-e^{3, k-3}+\cdots +(-1)^{\lfloor
    \frac{k}{2}\rfloor} e^{\lfloor \frac{k}{2}\rfloor, k-\lfloor
    \frac{k}{2}\rfloor}$.  To see this, note that in order to have
  $d^2=0$, by using formula (5), it is clear that any cocycle element
  has to include either the basis element $e^{1,k}$, and in this case
  this is a cocycle element, or it has to have one and only one
  element of type $e^{2,k}$ in the combination, and all those are
  listed above. The linear independence of the listed cocycle elements
  are clear.
\end{proof}

\paragraph{Example.} If $p=7$, then a basis for $\ker d^2$ is
\begin{align*}
  B(Z^2)=  \{e^{1,2}, e^{1,3}, e^{1,4}, e^{1,5}, e^{1,6},
  e^{1,7}, e^{2,3}, e^{2,5}-e^{3,4}, e^{2,7}-e^{3,6}+e^{4,5}\}.
\end{align*}
The cohomology $H^2(\m_0 ^\lambda (7))$ has a basis consisting of the
classes of
\[\{e^{1,7}, e^{2,3}, e^{2,5}-e^{3,4}, e^{2,7}-e^{3,6}+e^{4,5}\}.\]

\subsection{Restricted Cohomology with ${\bf \lambda}=0$}
If $\lambda=0$, then (\ref{induced2}) shows that $\ind^2=0$ so that
every ordinary 2-cocycle $\phi\in C^2(\m_0^0(p))$ gives rise to a
restricted 2-cocycle $(\phi,\tilde \phi)\in
C^2_*(\m_0^0(p))$. Therefore we can construct a basis for
$H^2_*(\m_0^0(p))$ from a basis for $H^2(\m_0^0(p))$ and the
restricted 2-cocycles $(0,\overline e_i)$, $1\le i\le p$. We summarize
this in the following theorem. As before, the case $p=2$ is treated
separately.

\begin{theorem}\label{zerolambda} Let $\lambda=0$.
  If $p=2$, then
  \[H^2_*(\m_0^0(2))\cong C^2_*(\m_0^0(2))=\spn(\{(0,\overline
    e^1),(0,\overline e^2), (e^{1,2},\tilde e^{1,2})\})\] is
  3-dimensional. If $p\ge 3$, then
  \[\dim(H^2_*(\m_0^0(p)))=\frac{3p+1}{2}\] and the cohomology
  classes of
  \[\{(0,\overline e^1),\dots, (0,\overline e^p),( e^{1,p},\tilde
    e^{1,p}),(\phi_5,\tilde \phi_5),(\phi_7,\tilde \phi_7),\dots,
    (\phi_{p+2},\tilde \phi_{p+2})\}\] form a basis.
\end{theorem}

\begin{proof} If $p=2$, the algebra $\m_0^0(2)$ is abelian, and
  $\ind^1=\ind^2=0$ so that $d^1_*=d^2_*=0$.

  If $p\ge 3$, the proof of Theorem~\ref{ordinary2} shows that
  \[\{e^{1,2},\dots , e^{1,p}, \phi_5,\phi_7,\dots, \phi_{p+2}\}\] is
  a basis for the kernel of $d^2$. Since $\lambda=0$, (\ref{induced2})
  implies $\ind^2=0$ so that
  \[\{(e^{1,2},\tilde e^{1,2}),\dots , (e^{1,p},\tilde e^{1,p}),
    (\phi_5,\tilde \phi_5),(\phi_7,\tilde \phi_7),\dots,
    (\phi_{p+2},\tilde\phi_{p+2})\}\] is a linearly independent subset
  of $\ker d^2_*$.  Moreover, for $1\le k\le p$, the maps
  $(0,\overline e^k)$ are also in the kernel of $d^2_*$ and the set
  \[B(Z^2_*)=\{(0, \overline e^1),\dots,(0, \overline e^p),
    (e^{1,2},\tilde e^{1,2}),\dots , (e^{1,p},\tilde e^{1,p}),
    (\phi_5,\tilde \phi_5),\dots, (\phi_{p+2},\tilde\phi_{p+2})\}\] is
  a basis $B(Z^2_*)$ for $\ker d^2_*$. We can redefine
  $\tilde e^{1,k-1}=\ind^1(e^{k})$ for $3\le k\le p$ without affecting
  the coboundary (see the remark at the end of Section 3) so that
  $d^1_*(e^k)=(e^{1,k-1},\tilde e^{1,k-1})$, and the set
  \[\{d^1_*(e^k)\ |\ 3\le k\le p\}\subseteq B(Z^2_*)\] forms a basis
  for the image $d^1_*(C^1(\m_0^0(p)))$.  It follows that
  \[\dim(\ker d^2_*)=p+\frac{3(p-1)}{2}=\frac{5p-3}{2},\]
  and $\dim(\im d^1_*)=p-2$ which completes the proof.
\end{proof}


\begin{remark} Another approach for determining the dimensions in
  Theorem~\ref{zerolambda} for $p\ge 2$ uses the six-term exact
  sequence in \cite{Ho}:
  \begin{diagram}[LaTeXeqno]
    \label{sixterm}
    0 &\rTo &H^1_*(\g,M)&\rTo &H^1(\g,M)&\rTo&\Hom_{\rm Fr}(\g,M^\g) & \rTo \\
    & \rTo & H^2_*(\g,M)&\rTo &H^2(\g,M)&\rTo&\Hom_{\rm
      Fr}(\g,H^1(\g,M)) &
  \end{diagram}
  If $\g=\m_0 ^\lambda (p)$ and $M=\F$, the map
  $\Delta: H^2(\g,M)\to\Hom_{\rm Fr}(\g,H^1(\g,M))$ in (\ref{sixterm})
  is given by
  \[\Delta_\phi(g)\cdot h = \phi(g,h^{[p]})=\ind^2(\phi,\omega)(g,h)\]
  where $\phi\in C^2(\g)$ and $g,h\in\g$ (see \cite{Viv}).  If
  $p\ge 3$, the map
  $H^1_*(\m_0 ^\lambda(p)) \to H^1(\m_0^\lambda (p))$ is an
  isomorphism so that the sequence (\ref{sixterm}) decouples and the
  sequence {\footnotesize
    \begin{diagram}[LaTeXeqno]
      \label{decoupled}
      0&\rTo&\Hom_{\rm Fr}(\m_0^\lambda,\F) & \rTo &
      H^2_*(\m_0^\lambda)&\rTo
      &H^2(\m_0^\lambda)&\rTo^\Delta&\Hom_{\rm
        Fr}(\m_0^\lambda,H^1(\m_0^\lambda))
    \end{diagram}} is exact. If $\lambda=0$, $\Delta=0$ so that the
  sequence (\ref{decoupled}) reduces to the short exact sequence
  \begin{diagram}[LaTeXeqno]
    \label{ses}
    0&\rTo&\Hom_{\rm Fr}(\m_0^\lambda,\F) & \rTo &
    H^2_*(\m_0^\lambda)&\rTo &H^2(\m_0^\lambda)&\rTo&0.
  \end{diagram}
\end{remark}
\subsection{Restricted Cohomology with ${\bf \lambda}\ne 0$}

If $\phi=\sum\sigma_{ij}e^{i,j}$ and
$(\phi,\omega)\in C^2_*(\m_0^\lambda)$, then (\ref{induced2}) shows
that
\[\ind^2(\phi,\omega)(e_j,e_i)=\lambda_i\sigma_{jp}.\]
Therefore, if $\lambda\ne 0$, then
$d^2_*(\phi,\omega)=(d^2\phi,\ind^2(\phi,\omega))=(0,0)$ if and only
if $d^2\phi=0$ and
$\sigma_{1p}=\sigma_{2p}=\cdots = \sigma_{p-1 p}=0$. This observation,
together with the calculation of the basis $B(Z^2_*)$ of $d^2_*$ in
the proof of Theorem~\ref{zerolambda}, proves the following

\begin{lemma}\label{lemma} Let $\lambda\ne 0$. If $p=2$, then
  $\ker d^2_* = \spn(\{(0,\overline e^1), (0,\overline e^2)\})$. If
  $p\ge 3$ and then a basis for the $\ker d^2_*$ is
  \[B(Z^2_*)-\{(e^{1,p},\tilde e^{1,p}),(\phi_{p+2},\tilde \phi_{p+2})
    \}.\]
\end{lemma}

As above, we treat the case $p=2$ separately. It is interesting to
note that when $\lambda\ne 0$, the basis for the cohomology
$H^2_*(\m_0 ^\lambda (2))$ depends on which coordinate of $\lambda$ is
non-zero, whereas this is not the case for $p\ge 3$.

\begin{theorem}\label{nonzerolambdapequal1} If $p=2$ and $\lambda=(\lambda_1,\lambda_2)\ne
  0$, then $\dim(H^2_*(\m_0 ^\lambda (2)))=1$.

  If $\lambda_2\ne 0$, then the cohomology class of the cocycle
  $(0,\overline e^1)$ is a basis for $H^2_*(\m_0 ^\lambda (2))$.

  If $\lambda_2=0$, then the cohomology class of the cocycle
  $(0,\overline e^2)$ is a basis for $H^2_*(\m_0 ^\lambda (2))$.
\end{theorem}

\begin{proof}
  If $p=2$, then the ordinary differentials $d^1=d^2=0$, and
  (\ref{induced1}) shows that $\ind^1(e^1)=0$ and
  $\ind^1(e^2)=\lambda_1\overline e^1+\lambda_2\overline e^2.$ It
  follows that
  \[\{(0, \lambda_1\overline e^1+\lambda_2\overline e^2)\}\] is a
  basis for the image $d^1_*(\m_0^\lambda(2))$. Moreover,
  $(\phi,\omega)\in\ker d^2_*$ if and only if $\ind^2(\phi,\omega)=0$
  if and only if
  $(\phi,\omega)\in\spn(\{(0,\overline e^1),(0,\overline e^2)\}$ by
  (\ref{induced2}). If $\lambda_2\ne 0$, then
  \[\ker(d^2_*)=\spn(\{(0,\overline e^1),(0,\overline
    e^2)\})=\spn(\{(0,\overline e^1),(0, \lambda_1\overline
    e^1+\lambda_2\overline e^2)\})\] so that the cohomology class of
  $(0,\overline e^1)$ is a basis for $H^2_*(\m_0 ^\lambda (2))$. If
  $\lambda_2=0$, then the image
  \[d^1_*(\m_0^\lambda(2))=\spn(\{(0,\lambda_1\overline
    e^1)\})=\spn(\{(0,\overline e^1)\})\] so that that the cohomology
  class of $(0,\overline e^2)$ is a basis for
  $H^2_*(\m_0 ^\lambda (2))$.
\end{proof}

\begin{theorem}\label{nonzerolambda}
  If $p\ge 3$ and ${\bf \lambda} \ne 0$, then
  \[\dim(H^2_*(\m_0 ^\lambda (p)))=\frac{3p-3}{2}\] and the
  cohomology classes of
  \[\{(0, \overline e^1),\dots,(0, \overline e^p), (\phi_5,\tilde
    \phi_5), (\phi_7,\tilde \phi_7)\dots,
    (\phi_{p},\tilde\phi_{p})\}\] form a basis. In particular, the
  cohomology $H^2_*(\m_0 ^\lambda (3))$ has a basis consisting of the
  cohomology classes of the cocycles
  \[\{(0, \overline e^1),(0, \overline e^2), (0, \overline e^3)\}.\]
\end{theorem}

\begin{proof} This follows immediately from Lemma~\ref{lemma}.
\end{proof}

\section{One-dimensional Central Extensions}

It is well known that one-dimensional central extensions of an
ordinary Lie algebra $\g$ are parameterized by the cohomology group
$H^2(\g)$ \cite{FuchsBook}. Likewise, restricted one-dimensional
central extensions of an ordinary Lie algebra $\g$ are parameterized
by the restricted cohomology group $H^2_*(\g)$ \cite{EvansFuchs2}. In
the case that $\g$ is a restricted simple Lie algebra, the authors in
\cite{EvFi} use the exact sequence (\ref{sixterm}) to show that
$H^2_*(\g)=H^2(\g)\oplus \Hom_{\rm Fr}(\g,\F)$ and the cohomology
classes of the cocyles $(0,\overline e^k)$, $1\le k\le \dim\g$ span a
($\dim\g$)-dimensional subspace of $H^*(\g)$. Moreover, if $E_k$
denotes the one-dimensional restricted central extension of $\g$
determined by the cohomology class of the cocycle $(0,\overline e^k)$,
then $E_k=\g\oplus \F c$ as a $\F$-vector space. For all
$1\le i,j\le \dim\g$,
\begin{align}\label{onedimext}
  \begin{split}
    [x_i,x_j] & =[x_i,x_j]_\g;\\
    [x_i, c] & = 0;\\
    e_{i}^{[p]} & = x_i^{[p]_\g}+ \delta_{k,i} c;\\
    c^{[p]} & = 0,
  \end{split}
\end{align}
where $[\cdot,\cdot]_\g$ and $\cdot^{[p]_\g}$ denote the Lie bracket
and $[p]$-operation in $\g$ respectively, and $\delta$ denotes the
Kronecker delta-function.

If $\lambda=0$, the restricted Lie algebra $\g=\m_0^0(p)$ is not
simple, but the the exact sequence (\ref{ses}) shows that Theorem~3.1
and and Corollary~3.2 in \cite{EvFi} also hold for the algebra
$\m_0^0(p)$, and we have

\begin{theorem} If $p\ge 2$, then
  \[H^2_*(\m_0^0(p))=H^2(\m_0^0(p))\oplus \Hom_{\rm
      Fr}(\m_0^0(p),\F),\] and there is a $p$-dimensional subspace of
  $H^2_*(\m_0^0(p))$ spanned by the cohomology classes of the cocycles
  $(0,\overline e^k)$ such that if $E_k$ denotes the corresponding
  one-dimensional restricted central extension, then
  $E_k=\m_0^0(p)\oplus \F c$ and the bracket and $[p]$-operator are
  given by (\ref{onedimext}).
\end{theorem}

If $\lambda\ne 0$, then an ordinary cocycle
$\phi\in C^2(\m_0^\lambda (p))$ need not give rise to a restricted
cocycle $(\phi,\omega)\in C^2_*(\m_0^\lambda (p))$. For example, if
$p=7$, then $\phi=e^{1,7}$ is an ordinary cocycle but (\ref{induced2})
shows that \[\ind^2(e^{1,7},\tilde e^{1,7})(e_1,e_j)=\lambda_j\] so
that $\ind^2\ne 0$ if $\lambda_j\ne 0$.

In any case, the sequence (\ref{ses}) shows that
$H^2_*(\m_0^\lambda (p))$ always has a $p$-dimensional subspace
spanned by the cohomology classes of the cocyles $(0,\overline e^k)$,
$1\le k\le p$. Each of the corresponding restricted one-dimensional
central extensions $E_k$ of $\m_0^\lambda (p)$ are trivial when
considered as ordinary one-dimensional central extensions.
\bibliography{references}{} \bibliographystyle{plain}

\end{document}